\documentclass[a4paper,11pt,english]{amsart}
\usepackage[applemac]{inputenc}
\usepackage[english]{babel}
\usepackage{amsfonts}
\usepackage{amsmath} 
\usepackage{amsthm}
\usepackage{fancyhdr}
\usepackage[hypertexnames=false]{hyperref}
\usepackage{latexsym}
\usepackage{mathrsfs}
\usepackage{array}
\usepackage{amssymb}
\usepackage{enumerate}
\usepackage{graphicx}
\usepackage{color}
\usepackage{stmaryrd}
\DeclareTextCommand{\textnu}{PU}{\83\275}
\newcommand{\ligne}{\vspace{1\baselineskip}}
\newcommand{\ph}{\phantomsection}
\newcommand{\cal}{\mathcal}

\newcommand{\R}{\mathbb  R}
\newcommand{\C}{\mathbb  C}
\newcommand{\N}{\mathbb  N}

\newcommand{\p}{\bf  p}

\newcommand{\W}{  \mathcal{W}   }

\newcommand{\eps}{\varepsilon}
\renewcommand{\epsilon}{\varepsilon}
\newcommand{\e}{  \text{e}   }
\newcommand{\wt}{  \widetilde   }
\newcommand{\Z}{  \mathbb{Z}   }

\renewcommand{\H}{  \mathcal{H}   }

\newcommand{\T}{  \mathcal{T} }

\newcommand{\dis}{\displaystyle}

\newcommand{\om}{  \omega   }
\newcommand{\ov}{  \overline  }
\renewcommand{\a}{  \alpha   }
\renewcommand{\b}{  \beta   }
\newcommand{\s}{  \sigma   }

\renewcommand{\phi}{  \varphi  }
\renewcommand{\L}{  \mathcal{L}   }

\newcommand{\<}{  \langle   }

\renewcommand{\>}{  \rangle   }

\numberwithin{equation}{section}
\theoremstyle{plain}

\newtheorem{theo}{Theorem}[section]
\newtheorem{lemm}[theo]{Lemma} 
\newtheorem{prop}[theo]{Proposition} 
 
\newtheorem{rema}[theo]{Remark}

\def\beq{\begin{equation}}   \def\eeq{\end{equation}}
\def\bea{\begin{eqnarray}}  \def\eea{\end{eqnarray}}

\renewcommand{\theequation}{\thesection.\arabic{equation}}
\newcounter{hran} \renewcommand{\thehran}{\thesection.\arabic{hran}}

\def\bmini{\setcounter{hran}{\value{equation}}
    \refstepcounter{hran}\setcounter{equation}{0}
    \renewcommand{\theequation}{\thehran\alph{equation}}\begin{eqnarray}}

\def\bminiG#1{\setcounter{hran}{\value{equation}}
\refstepcounter{hran}\setcounter{equation}{-1}
\renewcommand{\theequation}{\thehran\alph{equation}}
\refstepcounter{equation}\label{#1}\begin{eqnarray}}

\author{Pierre Germain}
\address{Courant Institute of Mathematical Sciences, 251 Mercer Street, New York 10012-1185 NY, USA}
\email{pgermain@cims.nyu.edu}
\author{Zaher Hani}
\address{School of Mathematics, Georgia Institute of Technology, Atlanta, GA 30332, USA}
\email{hani@math.gatech.edu}

\author{ Laurent Thomann }
\address{Laboratoire de Math\'ematiques J. Leray, UMR  6629 du CNRS, Universit\'e de Nantes, 
2, rue de la Houssini\`ere,
44322 Nantes Cedex 03, France}
\email{laurent.thomann@univ-nantes.fr}

\title[ Statistical study of the CR equation]{On the continuous resonant equation for NLS\\
II. Statistical study}

\begin{document}

\begin{abstract}
We consider the continuous resonant~(CR) system of the 2D cubic nonlinear Schr\"odinger (NLS) equation. This system arises in numerous instances as an effective equation for the long-time dynamics of NLS in confined regimes (e.g. on a compact domain or with a trapping potential). The system was derived and studied from a deterministic viewpoint in several earlier works \cite{FGH, HT, GHT1}, which uncovered many of its striking properties. This manuscript is devoted to a probabilistic study of this system. Most notably, we construct global solutions in negative Sobolev spaces, which leave Gibbs and white noise measures invariant. Invariance of white noise measure seems particularly interesting in view of the absence of similar results for NLS.
\end{abstract}

\thanks{P. G. is partially supported by NSF grant DMS-1101269, a start-up grant from the Courant Institute, and a Sloan fellowship.}
\thanks{L.T. is partially supported   by the  grant  ``ANA\'E'' ANR-13-BS01-0010-03}
\thanks{Z.~H. is partially supported by NSF Grant DMS-1301647, and a start-up fund
from the Georgia Institute of Technology.}
\maketitle

\section{Introduction}\label{intro}

\subsection{Presentation of the equation}

The purpose of this manuscript is to construct some invariant measures for the so-called continuous resonant~\eqref{CR} system of the cubic nonlinear Schr\"odinger equation. This system can be written as 

\begin{equation}\label{CR}\tag{CR}
\left\{
\begin{aligned}
&i\partial_{t}u=\T(u,u,u), \quad   (t,x)\in \R\times \R^2,\\
&u(0,x)=  f(x),
\end{aligned}
\right.
\end{equation}
where  the operator $\mathcal T$ defining the nonlinearity has several equivalent formulations corresponding to different interpretations/origins of this system. In its original formulation \cite{FGH} as the large-box limit\footnote{Starting with the equation on a torus of size $L$ and letting $L\to \infty$.} of the resonant cubic NLS\footnote{This is NLS with only the resonant interactions retained (a.k.a. first Birkhoff normal form). It gives an approximation of NLS for sufficiently small initial data.}, $\mathcal T$ can be written as follows: For $z\in \R^2$ and $(x_1, x_2)\in \R^2$, denoting by $x^\perp=(-x_2, x_2)$, we have
\begin{align*} 
\T(f_1,f_2,f_3)(z) & \overset{def}{=}  \int_\mathbb{R} \int_{\R^2}  f_1(x+z)f_2(\lambda x^{\perp}+z)\ov{f_3(x+\lambda x^{\perp}+z)} \,dx\, d\lambda.
\end{align*}
This integral can be understood as an integral over all rectangles having $z$ as a vertex. It has the following equivalent formulation \cite{GHT1}:
\begin{align*} 
\T(f_1,f_2,f_3) =2\pi \int_{\R}   \e^{-i\tau \Delta}  \Big[ (\e^{i\tau \Delta}f_1) ({ \e^{i\tau \Delta}f_2}) (\ov{\e^{i\tau \Delta}f_3})\Big] d\tau.
\end{align*}
It was shown in \cite{FGH} that the dynamics of~\eqref{CR} approximate that of the cubic NLS equation on a torus of size $L$ (with $L$ large enough) over time scales  $\sim L^2/\epsilon^2$ (up to logarithmic loss in $L$), where $\epsilon$ denotes the size of the initial data.

Another formulation of~\eqref{CR} comes from the fact that it is also the resonant system for the cubic nonlinear Schr\"odinger equation with harmonic potential given by:
\begin{equation}
\label{nls}
i \partial_t u - \Delta u + |x|^2 u = \mu |u|^2 u, \qquad  \mu=cst \in \R.
\end{equation} 
In this picture, $\T$ can be written as follows: Denoting by $H:=-\Delta +\vert x\vert^2=-\partial^2_{x_1}-\partial^2_{x_2}+x^2_1+x^2_2$ the harmonic oscillator on $ \R^2$, then 
\begin{equation*}
\T(f_1, f_2, f_3)=  2\pi  \int_{-\frac{\pi}4}^{\frac{\pi}4}   \e^{i\tau H}  \Big[ (\e^{-i\tau H}f_1) ({ \e^{-i\tau H}f_2}) (\ov{\e^{-i\tau H}f_3})\Big] d\tau.
\end{equation*}
As a result, the dynamics of~\eqref{CR} approximate the dynamics of \eqref{nls} over long nonlinear time scales for small enough initial data.

The equation~\eqref{CR} is Hamiltonian: Indeed, introducing  the functional
\begin{align*}
\mathcal{E}(u_1,u_2,u_3,u_4) & \overset{def}{=} \langle \mathcal{T}(u_1,u_2,u_3)\,,\,u_4 \rangle_{L^2}\\
& =  2\pi  \int_{-\frac{\pi}4}^{\frac{\pi}4}\int_{\R^2}  (\e^{-it H}u_1)   (\e^{-it H}u_2) (\ov{ \e^{-it H}u_3}) (\ov{\e^{-it H}u_4}) dx\, dt,\nonumber
\end{align*}
and setting 
$$\mathcal{E}(u) := \mathcal{E}(u,u,u,u),$$
then~\eqref{CR}  derives from the Hamiltonian $\mathcal E$ given the symplectic form  $\omega(f,g) = -4 \mathfrak{Im} \langle f\,,\, g \rangle_{L^2(\mathbb{R}^2)}$ on~$L^2(\mathbb{R}^2)$, so that~\eqref{CR} is equivalent to 
\begin{equation*}
i\partial_t f = \frac{1}{2} \frac{\partial \mathcal{E}(f)}{\partial \bar f}.
\end{equation*}

In addition to the two instances mentioned above in which~\eqref{CR} appears to describe the long-time dynamics of the cubic NLS equation -- with or without potential -- we mention the following: 

\begin{itemize}
\item[$\bullet$] The equation~\eqref{CR} appears as a modified scattering limit of the cubic NLS on $\R^3$ with harmonic tapping in two directions. Here,~\eqref{CR} appears as an asymptotic system and any information on the asymptotic dynamics of~\eqref{CR} directly gives the corresponding behavior for NLS with partial harmonic trapping. We refer to Hani-Thomann \cite{HT} for more details.
\item When restricted to the Bargmann-Fock space (see below), the equation~\eqref{CR} turns out to be the Lowest-Landau-Level equation, which describes fast rotating Bose-Einstein condensates (see~\cite{ABN,Nier,GGT}).
\item[$\bullet$] The equation~\eqref{CR} can also be interpreted as describing the effective dynamics of high frequency envelopes for NLS on the unit torus $\mathbb T^2$. This means that if the initial data $\phi(0)$ for NLS has its Fourier transform given by\footnote{Up to a normalizing factor in $H^s$, $s>1$.} $\{\widehat \varphi(0, k)\sim g_0(\frac{k}{N})\}_{k \in \Z^2}$, and if $g(t)$ evolves according to~\eqref{CR} with initial data $g_0$ and $\varphi(t)$ evolves according to NLS with initial data $\phi(0)$, then $g(t, \frac{k}{N})$ approximates the dynamics of $\widehat{\varphi}(t, k)$ in the limit of large $N$ (see \cite[Theorem 2.6]{FGH}). 
\end{itemize}
 
\subsection{Some properties and invariant spaces} We review some of the properties of the~\eqref{CR} equation that will be useful in this paper. For a more detailed study of the equation we refer to~\cite{FGH, GHT1}.

First,~\eqref{CR} is globally well-posed in $L^2(\R^2)$. Amongst its conserved quantities, we note
$$
\int_{\mathbb{R}^2} |u|^2 \,dx, \quad \int_{\mathbb{R}^2} (|x|^2 |u|^2 + |\nabla u|^2)\,dx = \int_{\mathbb{R}^2}  \overline{u} Hu\,dx,
$$
(recall that $H$ denotes the harmonic oscillator $H = -\Delta + |x|^2$). This equation also enjoys many invariant spaces, in particular:
\begin{itemize}
\item[$\bullet$] The eigenspaces $(E_N)_{N\geq 0}$ of the harmonic oscillator  are stable (see \cite{FGH, GHT1}). This is a manifestation of the  fact that~\eqref{CR} is the resonant equation associated to \eqref{nls}. Recall that $H$ admits a complete basis of eigenvectors for $L^2(\R^2)$; each eigenspace $E_N$ $(N=0, 1,2,...)$ has dimension $N+1$. 

\item[$\bullet$] The set of radial functions is stable, as follows from the invariance of $H$ under rotations (see\;\cite{GHT1}). Global dynamics on $L^2_{rad}(\mathbb{R}^2)$, the radial functions of $L^2(\mathbb{R}^2)$, can be defined. A basis of normalized eigenfunctions of $H$ for $L^2_{rad}(\mathbb{R}^2)$ is given by,
$$
\mbox{if $n \in \mathbb{N}$}, \quad \phi_n^{rad}(x) = \frac{1}{\sqrt{\pi}} L_n^{(0)}(|x|^2) e^{-\frac{|x|^2}{2}} \quad \mbox{with} \quad L_k^{(0)}(x) = e^x \frac{1}{n!}\left( \frac{d}{dx} \right)^k(e^{-x}x^k).
$$
We record that $H \phi_n^{rad} = (4n+2) \phi_k^{rad}$.
\item[$\bullet$] If $\mathcal{O}(\C) $ stands for the set of entire functions on $\C$ (with the identification $z=x_1+ix_2$), the Bargmann-Fock space $L^2_{hol}(\mathbb{R}^2) = L^{2}(\R^2)\cap (\mathcal{O}(\C) \e^{-\vert z \vert^2/2})$ is invariant by the flow of~\eqref{CR}. Global dynamics on $L^2_{hol}(\mathbb{R}^2)$ can be defined.
A basis of normalized eigenfunctions of~$H$ for\;$L^2_{hol}(\mathbb{R}^2)$ is given by the ``holomorphic" Hermite functions, also known as the ``special Hermite functions", namely
$$
\mbox{if $n \in \mathbb{N}$}, \quad \phi_n^{hol}(x) = \frac{1}{\sqrt{\pi n!}}(x_1+i x_2)^n \e^{-\vert x\vert^2/2 }.
$$
Notice that $H \phi^{hol}_n=2(n+1)\phi^{hol}_n$. It is proved in~\cite{GHT1} that 
\begin{equation} \label{tfi}
\T(\phi^{hol}_{n_1},\phi^{hol}_{n_2},\phi^{hol}_{n_3}) = \alpha_{n_1,n_2,n_3,n_4} \phi^{hol}_{n_4}, \qquad n_4=n_1+n_2-n_3,
\end{equation} 
with
\begin{equation*} 
\alpha_{n_1,n_2,n_3,n_4} = \H(\phi^{hol}_{n_1},\phi^{hol}_{n_2},\phi^{hol}_{n_3},\phi^{hol}_{n_4}) =\frac{\pi }{8} \frac{(n_1+n_2)!}{2^{n_1+n_2}\sqrt{n_1 !n_2 !n_3 !n_4 !}}{\bf 1}_{n_1+n_2=n_3+n_4}.
\end{equation*} 

As a result, the~\eqref{CR} system reduces to the following infinite-dimensional system of ODE when restricted to Span$\{\phi_n\}_{n \in \N}$:
$$
i\partial_t c_n(t)=\sum_{\substack{n_1, n_2, n_3 \in \N\\ n_1+n_2-n_3=n}} \alpha_{n_1,n_2,n_3,n} c_{n_1}(t)c_{n_2}(t)\overline{c_{n_3}}(t).
$$
\end{itemize}

\subsection{Statistical solutions}
In this paper we construct global probabilistic solutions on each of the above-mentioned spaces which leave invariant either Gibbs or white noise measures. More precisely, our main results can be summarized as follows:
 
\begin{itemize}
\item[$\bullet$]  We construct global strong  flows  on
  $$X_{rad}^0(\R^2) =\cap_{\s>0}  \H^{-\s}_{rad}(\R^2),$$
    and on 
  \begin{equation*}
X_{hol}^{0}(\R^2) := \big( \cap_{\sigma>0} \H^{-\sigma} (\R^2)\big)\cap (\mathcal{O}(\C)\e^{-\vert z \vert^2/2}),
\end{equation*}
which leave the Gibbs measures invariant (see Theorem \ref{thm1}).
\medskip 
 \item[$\bullet$]  
 We construct  global weak probabilistic solutions on 
  \begin{equation*}
X_{hol}^{-1}(\R^2) := \big( \cap_{\sigma>1} \H^{-\sigma} (\R^2)\big)\cap (\mathcal{O}(\C)\e^{-\vert z \vert^2/2}),
\end{equation*}
and this dynamics  leaves the white noise measure invariant (see Theorem \ref{thm2}).
 \end{itemize}

Since the 90's,  there have been many works devoted to  the construction of Gibbs measures  for dispersive equations, and more recently, much attention has been paid to the well-posedness of these equations  with random initial conditions. We refer to the introduction of \cite{PRT2} for references on the subject. In particular, concerning the construction of strong solutions for the nonlinear harmonic oscillator (which is related to~\eqref{CR}), we refer to \cite{tho,BTT,Deng,poiret1,poiret2,PRT2}.

\medskip 

Construction of flows invariant by white noise measure is much trickier due to the low regularity of the support of such measures, and there seems to be no results of this sort for NLS equations. We construct weak solutions on the support of the white noise measure on $X_{hol}^{-1}(\R^2)$ using a method based on a compactness argument in the space of measures (the Prokhorov theorem) combined with a representation theorem of random variables (the Skorohod theorem). This approach has been first applied to the Navier-Stokes and Euler equations in Albeverio-Cruzeiro \cite{AC} and Da Prato-Debussche~\cite{DPD} and extended to dispersive equations by Burq-Thomann-Tzvetkov \cite{BTT2}. We refer to \cite{BTT2} for a self-contained presentation of the method.

\subsection{Notations} Define the harmonic Sobolev spaces for $s \in \mathbb{R}$, $p\geq 1$ by 
\begin{equation*} 
\W^{s, p}= \W^{s, p}(\R^2) = \big\{ u\in L^p(\R^2),\; {H}^{s/2}u\in L^p(\R^2)\big\},    \qquad  \H^{s}= \W^{s, 2}.
\end{equation*}
They are endowed with the natural norms $\Vert u\Vert_{\W^{s,p}}$. Up to equivalence of norms we have for $s \geq 0$, $1<p<+\infty$ (see \cite[Lemma~2.4]{YajimaZhang2}) 
\begin{equation}\label{equiv}
\Vert u\Vert_{\W^{s,p}} = \Vert  H^{s/2}u\Vert_{L^{p}} \equiv \Vert (-\Delta)^{s/2} u\Vert_{L^{p}} + \Vert\<x\>^{s}u\Vert_{L^{p}}.
\end{equation}

Consider a   probability space $(\Omega, {\cal F}, {\bf p})$. In all the paper,  $\{g_{n},\;n\geq0\}$  and   $\{g_{n,k},\;n\geq0,\, 0\leq k\leq n\}$ are independent standard complex Gaussians $\mathcal{N}_{\C}(0,1)$ (their probability density function reads thus $\frac{1}{\pi} e^{-|z|^2} dz$, $dz$ being Lebesgue measure on $\mathbb{C}$). If $X$ is a random variable, we denote by $\mathscr{L}(X)$ its law (or distribution).

We will sometimes use the notations $L^{p}_{T}=L^{p}(-T,T)$ for $T>0$.  If $E$ is a Banach space and  $\mu$ is a measure on $E$, we write  $L^{p}_{\mu}=L^{p}(\text{d}\mu)$ and $\|u\|_{L^{p}_{\mu}E}=\big\|\|u\|_{E}\big\|_{L^{p}_{\mu}}$. We define $X^{\s}(\R^2)=\bigcap_{\tau<\s} \H^{\tau}(\R^2)$, and if $I\subset \R$ is an interval,  with an abuse of notation, we write $\mathcal{C}\big(I; X^{\s}(\R^2)\big)=\bigcap_{\tau<\s}\mathcal{C}\big(I; \H^{\tau}(\R^2)\big)$. 

Finally, $\N$ denotes the set of natural integers including 0; $c,C>0$ denote constants the value of which may change from line to line. These constants will always be universal, or uniformly bounded with respect to the other parameters. For two quantities $A$ and $B$, we denote $A \lesssim B$ if $A \leq CB$, and $A \approx B$ if $A \lesssim B$ and $A \gtrsim B$.

\section{Statement of the results}
 
 As mentioned above, we will construct strong solutions on the support of Gibbs measures and prove the invariance of such measures. For white noise measures, solutions are weak and belong to the space $C_T X^{-1}$. We start by discussing the former case.

\subsection{Global strong solutions invariant by Gibbs measure} 

\subsubsection{\bf Measures and dynamics on the space $\boldsymbol{E_N}$}  The operator $H$ is   self-adjoint   on  $L^2(\R^2)$, and    has the discrete  spectrum
$\{2N +2, N \in \N\}$. For $N\geq 0$, denote by $E_N$ the eigenspace  associated to the eigenvalue $2N+2$. This space  has dimension $N+1$. Consider  $(\phi_{N,k})_{0\leq k\leq N} $ any orthonormal basis of\;$E_N$. Define  $\gamma_N\in L^2(\Omega; E_N)$ by
\begin{equation*} 
\gamma_N(\om,x)=\frac{1}{\sqrt{N+1}}\sum_{k=0}^{N}g_{N,k}(\omega)\phi_{N,k}(x).
\end{equation*}
The distribution of the random variable $\gamma_{N}$ does not depend on the choice of the basis, and observe that the law of large numbers gives
\begin{equation*}
\Vert \gamma_N \Vert^2_{L^2(\R^2)}=\frac{1}{N+1}\sum_{k=0}^{N}\vert g_{N,k}(\omega)\vert^2\longrightarrow 1\quad \text{a.s. when}\quad N\longrightarrow+\infty.
\end{equation*}
Then we define the  probability  measure $\mu_N=\gamma_{\#} {\bf p}:={\bf p}\circ \gamma^{-1}_N$ on $E_N$.\medskip

The  $L^p$ properties of the measures $\mu_N$ have been studied in  \cite{PRT1} with an improvement in \cite{RT}. We mention in particular the following result 

\begin{theo}[{\cite{PRT1,RT}}] \label{thm21}
There exist $c,C_{1},C_{2}>0$  such that for all $N\geq N_0$
\begin{equation*} 
\mu_N \left[u\in E_N :  C_{1} N^{-1/2}(\log N)^{1/2}  \Vert u\Vert_{L^2(\R^{2})}\leq \Vert u\Vert_{L^{\infty}(\R^{2})} 
\leq C_{2}  N^{-1/2}(\log N)^{1/2} \Vert u\Vert_{L^2(\R^{2})}\,  \right] \geq 1- N^{-c}.
\end{equation*}
\end{theo} 
 
This proposition is a direct application of \cite[Theorem 3.8]{RT} with $h=N^{-1}$ and $d=2$. Notice that for all $u\in E_N$, we have $\Vert u\Vert_{\H^s} =(2N+2)^{s/2}\Vert u\Vert_{L^2}$. The best (deterministic) $L^{\infty}$-bound for an eigenfunction $u\in E_N$ is given by Koch-Tataru\;\cite{KoTa} 
\begin{equation}\label{det}
 \Vert u\Vert_{L^{\infty}(\R^{2})}  \leq C  \Vert u\Vert_{L^2(\R^{2})},
\end{equation}
and the previous estimate is optimal, since it is saturated by the radial Hermite functions. Therefore the result of Theorem\;\ref{thm21} shows that there is almost a gain of one derivative compared to the deterministic estimate\;\eqref{det}.\medskip

It turns out that the measures $\mu_N$ are invariant under the flow of~\eqref{CR}, and we have the following result.
 
\begin{theo}\label{thmN}
For all $N\geq 1$, the measure $\mu_N$ is invariant under the flow $\Phi$ of~\eqref{CR} restricted to\;$E_N$. Therefore, by the Poincar\'e theorem, $\mu_N$-almost all $u\in E_N$ is recurrent in the following sense: for $\mu_N$-almost all $u_0\in E_N$ there exists a sequence of times $t_n\longrightarrow +\infty$ so that 
\begin{equation*}
\lim_{n\to +\infty} \big\Vert  \Phi({t_n})u_0-u_0 \big\Vert_{L^2(\R^2)}=0.
\end{equation*}
\end{theo} 

In the previous result, one only uses the invariance of  the probability measure $\mu_N$ by the flow, and no additional property of the equation~\eqref{CR}.

\subsubsection{\bf Gibbs Measure on the space $\boldsymbol X_{}^0( \boldsymbol \R^2)$ and a well-posedness result} 
 
In the sequel we either consider the family  $(\phi_{n}^{rad})_{n\geq 0}$ of the radial Hermite functions, which are eigenfunctions of $H$ associated to the eigenvalue 
$\lambda_n^{rad} = 4n+2$; or the family  $(\phi_{n}^{hol})_{n\geq 0}$ of the holomorphic Hermite functions, which are eigenvalues of $H$ associated to the eigenvalue $\lambda_n^{hol} = 2n+2$. Set
$$X_{rad}^0(\R^2) =\cap_{\s>0}  \H^{-\s}_{rad}(\R^2),$$
\begin{equation*}
X_{hol}^{0}(\R^2) := \big( \cap_{\sigma>0} \H^{-\sigma} (\R^2)\big)\cap (\mathcal{O}(\C)\e^{-\vert z \vert^2/2}).
\end{equation*}
In the following, we denote $X_{\star}^{0}(\R^2)$ for $X_{rad}^{0}(\R^2)$ or $X_{hol}^{0}(\R^2)$, $\phi_{n}^{\star}$ for $\phi_{n}^{rad}$ or $\phi_{n}^{hol}$, etc...

Define now $\gamma_{\star} \in L^2(\Omega; X_{\star}^0(\R^2) )$ by
\begin{equation*} 
\gamma_{\star}(\om,x)=\sum_{n=0}^{+\infty}  \frac{g_{n}(\omega)}{\sqrt{\lambda_n^\star}}\phi^\star_{n}(x),
\end{equation*}
and consider the Gaussian  probability  measure $\mu_{\star}=(\gamma_\star)_{\#} {\bf p}:={\bf p}\circ \gamma_{\star}^{-1}$.

\begin{lemm}\ph \label{lema}
In each of the previous cases, the measure $\mu_\star$  is a probability measure on $X_{\star}^0(\R^2)$.
\end{lemm}
 
Notice that since~\eqref{CR} conserves the $\mathcal{H}^1$ norm, $\mu_\star$ is formally invariant by its flow. More generally, we can define a family $(\rho_{\star,\beta})_{\beta\geq 0}$ of  probability  measures on $X_{\star}^0(\R^2)$ which are formally invariant by~\eqref{CR} in the following way: define for $\beta\geq 0$ the measure $\rho_{\star}=\rho_{\star,\b}$ by
 \begin{equation}\label{defrho}
 d\rho_{\star}(u)=C_{\b}\e^{-\beta \mathcal {E}(u)}d\mu_{\star}(u),
 \end{equation}
 where $C_\b>0$ is a normalising constant. In Lemma \ref{lem.born}, we will show that $\mathcal E(u)<+\infty$, $\mu_{\star}$ a.s., which enables us to define this probability measure.

For all $\beta\geq 0$, $\rho_{\star}\big(X_{\star}^0(\R^2)   \big)=1$ and  $\rho_{\star}\big(L^2_{\star}(\R^2)   \big)=0$.
\begin{rema}
 Observe that we could also give a sense to a generalised version of \eqref{defrho} when $\beta<0$ using the renormalizing method of Lebowitz-Rose-Speer. We do not give the details and refer to \cite{BTT} for such a construction.
 \end{rema}

We are now able to state the following global existence result.

\begin{theo}\label{thm1} 
Let $\b\geq 0$.  There exists a set $\Sigma \subset X_{\star}^{0}(\R^{2})$ 
of full $\rho_\star$ measure so that for every $f\in \Sigma$ the
equation~\eqref{CR} with initial condition $u(0)=f$ has a unique global solution $u(t)=\Phi(t)f$ such that for any $0<s<1/2$
$$
u(t)-f\,\in\,\mathcal{C}\big( \R; \H^s(\R^2) \big).
$$
Moreover, for all $\s>0$ and $t\in \R$
\begin{equation*}
\|u(t)\|_{\H^{-\s}(\R^2)}\leq C\bigl( \Lambda(f,\s)+\ln^{\frac12} \big(1+|t|\big)\bigr),
\end{equation*}
and the constant $\Lambda(f,\s)$ satisfies the bound 
$
\mu_{\star}\big(f :\Lambda(f,\s)>\lambda \big) \leq C\e^{-c\lambda^{2}}.
$
\item Furthermore, the measure $\rho_\star$ is invariant by $\Phi$: For any $\rho_\star$ measurable set $A\subset \Sigma$, for any $t\in \R$, $\rho_\star(A)=\rho_\star(\Phi(t)(A))$.
 \end{theo}

\subsection{White noise measure on the space \texorpdfstring{$ \boldsymbol X_{\boldsymbol hol}^{-1}(\R^{2})$}{X(\R\texttwosuperior)} and weak solutions} 
Our aim is now to construct weak solutions on the support of the white noise measure. 
Consider the   Gaussian   random variable
  \begin{equation}\label{def.phi}
\gamma(\om,x)=\sum_{n=0}^{+\infty}  g_{n}(\omega) \phi_{n}^{hol}(x)=\frac{1}{\sqrt{\pi }}\Big(\,\sum_{n=0}^{+\infty}  \frac{{(x_1+i x_2)^ng_{n}(\omega)}}{\sqrt{ n!}}\,\Big)  \e^{-\vert x\vert^2/2 },
\end{equation}
 and the measure $\mu={\bf p}\circ \gamma^{-1}$. As in Lemma \ref{lema} we can show that the measure $\mu$  is a probability measure on 
 \begin{equation*}
X_{hol}^{-1}(\R^2) := \big( \cap_{\sigma>1} \H^{-\sigma} (\R^2)\big)\cap (\mathcal{O}(\C)\e^{-\vert z \vert^2/2}).
\end{equation*}
 
Since $\Vert u\Vert_{L^2(\R^2)}$ is preserved by~\eqref{CR}, $\mu$ is formally invariant under~\eqref{CR}.  We are not able to define a flow at this level of regularity, however using compactness arguments combined with probabilistic methods, we will construct weak solutions.
\begin{theo}\ph\label{thm2} 
There exists a set $\Sigma \subset X_{hol}^{-1}(\R^{2})$ 
of full $\mu$ measure   so that for every $f\in \Sigma$ the
equation~\eqref{CR} with 
initial condition $u(0)=f$ has a   solution 
\begin{equation*}
u\in \bigcap_{\s>1}\mathcal{C}\big(\R\, ; \H^{-\s}(\R^2) \big).
\end{equation*}
The distribution of the random variable $u(t)$ is equal to  $\mu$ (and thus independent of $t\in \R$):
\begin{equation*}
\mathscr{L}_{X^{-1}(\R^2)}\big(u(t)\big)=\mathscr{L}_{X^{-1}(\R^2)}\big(u(0)\big)=\mu,\quad \forall\,t\in \R.
\end{equation*} 
\end{theo}

\begin{rema}
 One can also define the Gaussian measure $\mu={\bf p}\circ \gamma^{-1}$ on $X_{}^{-1}(\R^2) =\cap_{\s>1}  \H_{}^{-\s}(\R^2)$ by
  \begin{equation*} 
\gamma(\om,x)=\sum_{n=0}^{+\infty} \frac{1}{\sqrt{\lambda_n}} \sum_{k=-n}^{n}  {g_{n,k}(\omega)} \phi_{n,k}(x),\quad \lambda_n=2n+2
\end{equation*}
(where the $\phi_{n,k}$ are an orthonormal basis of eigenfunctions of the harmonic oscillator and the angular momentum operator). Since $\Vert u\Vert_{\H^1(\R^2)}$ is preserved by~\eqref{CR}, $\mu$ is formally invariant under~\eqref{CR}, but we are not able to obtain an analogous result in this case.

The same comment holds for  the white noise measure $\mu={\bf p}\circ \gamma^{-1}$ on $X_{rad}^{-1}(\R^2) =\cap_{\s>1}  \H_{rad}^{-\s}(\R^2)$ with 
\begin{equation*} 
\gamma(\om,x)=\sum_{n=0}^{+\infty}  g_{n}(\omega) \phi_{n}^{rad}(x),
\end{equation*}
which is also formally invariant under~\eqref{CR}.
 \end{rema}

\subsection{Plan of the paper}

The rest of the paper is organised as follows. In Section \ref{sect3} we prove the results concerning the strong solutions, and in Section \ref{sect4} we construct the weak solutions.
 
\section{Strong solutions}\label{sect3}
 
\subsection{ Proof of Theorem \ref{thmN}} 
 
The proof of Theorem \ref{thmN} is an application of the Liouville theorem. Indeed, write ${u_N=\sum_{k=0}^{N}c_{N,k}\phi_{N,k}\in E_N}$, then 
\begin{equation*}
d\mu_N= \frac{(N+1)^{N+1}}{\pi^{N+1}} \exp\Big({-\dis (N+1)\sum_{k=0}^{N}\vert c_{N,k}\vert^2}\Big)\prod_{k=0}^{N} d a_{N,k}\,d b_{N,k},
\end{equation*}

where $c_{N,k} = a_{N,k} + i b_{N,k}$. 
 
 The Lebesgue measure $\prod_{k=0}^{N} d a_{N,k}db_{N,k}$ is preserved since~\eqref{CR} is Hamiltonian and ${\sum_{k=0}^{N}\vert c_{N,k}\vert^2=\Vert u_N\Vert^2_{L^2}} $ is a constant of motion.

\subsection{ Proof of Theorem \ref{thm1}} 

We start with the proof of Lemma \ref{lema}.
 
\begin{proof}[Proof of Lemma \ref{lema}] We only consider the case $X_{\star}^{0}(\R^2)=X_{hol}^{0}(\R^2)$.  It is enough to show that      $\gamma_{hol}\in X_{hol}^{0}(\R^2)$, $\p$-a.s. First,
for all $\s>0$ we have 
\begin{equation}\label{est.phi}
\int_{\Omega}\Vert \gamma_{hol}   \Vert^2_{\H^{-\s}{(\R^2)}} d{\bf p}(\om)=\int_{\Omega} \sum_{n=0}^{+\infty} \frac{ \vert{g_{n}}\vert^2}{ (\lambda_n^{hol})^{\s+1}}    d{\bf p}(\om) =C \sum_{n=0}^{+\infty} \frac{1}{{(n+1)}^{\s+1}}<+\infty,
\end{equation}
therefore   $ \gamma_{hol} \in     \bigcap_{\s>0} L^2\big(\Omega\,; \, \H^{-\s}(\R^2)\big)$. Next, by \cite[Lemma 3.4]{CO},   for all $A\geq 1$ there exists a set $\Omega_A\subset \Omega$ such that ${\p} (\Omega^c_A)\leq \exp{(-A^{\delta})}$ and for all $\om \in \Omega_A$, $\eps>0$, $n\geq 0$
\begin{equation*}
\vert g_n{(\om)}\vert \leq C A (n+1)^{\eps}.
\end{equation*}
 Then for $  \om \in \bigcup_{A\geq 1} \Omega_A$, \;$\dis \sum_{n=0}^{+\infty}  \frac{{z^ng_{n}(\omega)}}{\sqrt{ \lambda_{n}^{hol}\,n!}} \in \mathcal{O}(\C)$.
\end{proof}

We first define a smooth version of the usual spectral projector. Let $\chi\in \mathcal{C}_{0}^{\infty}(-1,1)$,  so that $0\leq \chi\leq 1$, with $\chi=1$ on $[-\frac12,\frac12]$. We define  the operators $S_{N}=\chi\big(\frac{H}{\lambda_{N}}\big)$ as 
 
\begin{equation*} 
S_{N}\big(\sum_{n=0}^{\infty}c_n \phi^\star_{n}\big)=\sum_{n=0}^{\infty}\chi\big(\frac{\lambda_{n}^\star}{\lambda_{N}^\star}\big)c_n \phi_{n}^\star.
\end{equation*}
Then for all $1<p<+\infty$, the operator $S_N$ is bounded in $L^p(\R^2)$ (see \cite[Proposition 2.1]{Deng} for a proof).

\subsubsection{ \bf Local existence}
 It will be useful to work with an approximation of~\eqref{CR}.  We consider the dynamical system given by the Hamiltonian $\H_{N}(u):=\H(S_{N}u)$. This system reads
   \begin{equation}\label{eqtN} 
\left\{
\begin{aligned}
&i\partial_{t}u_N=\T_N(u_N), \quad   (t,x)\in \R\times \R^2,\\
&u_{N}(0,x)=f,
\end{aligned}
\right.
\end{equation}
 and  $\T_N(u_{N}):=S_N \T(S_N u,S_N u,S_N u)$. Observe that \eqref{eqtN} is a finite dimensional dynamical system on $\bigoplus_{k=0}^{N} E_{k}$ and that the projection of $u_{N}(t)$ on its complement is constant. For $\beta\geq 0$ and $N\geq 0$ we define the measures  $\rho^{N}_{\star}$ by
 \begin{equation*} 
 d\rho^{N}_{\star}(u)=C^{N}_{\b}\e^{-\beta \H_{N}(u)}d\mu_{\star}(u),
 \end{equation*}
 where $C^{N}_\b>0$ is a normalising constant. We have the following result 
 
 \begin{lemm}
 The system \eqref{eqtN} is globally well-posed in $L^{2}(\R^{2})$. Moreover, the measures  $\rho^{N}_{\star}$ are invariant by its flow denoted by $\Phi_{N}$.
  \end{lemm}
 
  \begin{proof}
  The global existence follows from the conservation of $\|u_{N}\|_{L^{2}(\R^{2})}$. The invariance of the measures is a consequence of the Liouville theorem and the conservation of $\sum_{k=0}^\infty \lambda_k |c_k|^2$ by the flow of~\eqref{CR} (see~\cite{FGH}). We refer to \cite[Lemma 8.1 and Proposition 8.2]{BTT} for the details.
    \end{proof}
  
  We now state a result concerning dispersive bounds of Hermite functions
  
  \begin{lemm}\label{lem.born}
For all $2\leq p\leq +\infty$, 
\begin{equation}\label{est1}
\Vert \phi_n^{hol}\Vert_{L^p(\R^d)}\leq  C n^{\frac1{2p}-\frac14},
\end{equation}
\begin{equation}\label{est2}
\|\phi^{rad}_n\|_{L^4(\R^d)}\leq C n^{-\frac{1}{4}} (\ln n)^{\frac14}.
\end{equation}
\end{lemm}
  
  \begin{proof}
By Stirling, we easily get that $ \Vert \phi_n^{hol}\Vert_{L^{\infty}(\R^d)}\leq  C n^{-\frac14}$, which is~\eqref{est1} for $p=\infty$; the estimate for $2 \leq p \leq \infty$ follows by interpolation. For the the proof of \eqref{est2}, we refer to \cite[Proposition 2.4]{IRT}.
  \end{proof}

  \begin{lemm}\label{ld}
  \begin{enumerate}[(i)]
\item We have
\begin{multline}\label{b1}
\exists C>0,\exists c>0, \forall \lambda\geq 1, \forall N\geq 1,\\
\mu_{\star}\big(\,u\in X^{0}_{\star}(\R^{2}): \|e^{-itH}S_Nu\|_{L^{4}([-\frac{\pi}4,\frac{\pi}4]\times \R^{2})}>\lambda\,\big)\leq Ce^{-c\lambda^2}\,.
\end{multline}
\item There exists $\beta>0$ such that
\begin{multline}\label{b2}
\exists\, C>0,\exists\, c>0,\,\, \forall \lambda\geq 1,\, \forall N\geq N_0\geq 1,
\\
\mu_{\star}\big(\,u\in X^{0}_{\star}(\R^{2})\,:\,\|e^{-itH}(S_N-S_{N_{0}})u\|_{L^{4}([-\frac{\pi}4,\frac{\pi}4]\times \R^{2})}>\lambda\,\big)\leq 
Ce^{-cN_0^{\beta}\lambda^2}\,.
\end{multline}
\item In the holomorphic case: for all $2\leq p <+\infty$  and $s< \frac12-\frac1p$
\begin{equation}\label{b3}
\begin{split}
& \exists C>0,\exists c>0, \forall \lambda\geq 1, \forall N\geq 1,\\
& \quad \quad \mu_{hol}\big(\,u\in X^{0}_{hol}(\R^{2}): \|e^{-itH}u\|_{L_{[-\frac{\pi}4, \frac{\pi}4]}^p\W^{s,p}(\R^2)}>\lambda\,\big)\leq Ce^{-c\lambda^2},\\
& \quad \quad \mu_{hol}\big(\,u\in X^{0}_{hol}(\R^{2}): \|e^{-itH}u\|_{L^{8/3}([-\frac{\pi}4,\frac{\pi}4]\times \R^2)}>\lambda\,\big)\leq Ce^{-c\lambda^2}\,.
\end{split}
\end{equation}
\item In the radial case: for all $s<1/2$
\begin{multline}\label{b4}
\exists C>0,\exists c>0, \forall \lambda\geq 1, \forall N\geq 1,\\
\mu_{rad}\big(\,u\in X^{0}_{rad}(\R^{2}): \|e^{-itH}u\|_{L_{[-\frac{\pi}4,\frac{\pi}4]}^4\W^{s,4}(\R^2)}>\lambda\,\big)\leq Ce^{-c\lambda^2}\,.
\end{multline}
\end{enumerate}
\end{lemm}

\begin{proof}  We have that
\begin{multline*}
\mu_{\star}\big(\,u\in X^{0}_{\star}(\R^{2})\,:\, \|e^{-itH}S_Nu\|_{L^{4}([-\frac{\pi}4,\frac{\pi}4]\times \R^{2})}>\lambda\,\big)\\
={\bf p}\left( \Big\|\sum_{n=0}^{\infty}{e^{-it\lambda_{n}}\chi\bigl(\frac{\lambda_{n}}{\lambda_{N}}\bigr)\frac{g_{n}(\omega)}{\sqrt{\lambda_{n}}}}\varphi_n^\star(x)\Big\|_{L^{4}([-\frac{\pi}4,\frac{\pi}4]\times \R^{2})}>\lambda \right)
\end{multline*}
Set 
$$
F(\omega,t,x)\equiv\sum_{n=0}^{\infty}{e^{-it\lambda_{n}^\star}\chi\bigl(\frac{\lambda_{n}^\star}{\lambda_{N}^\star}\bigr)\frac{g_{n}(\omega)}{\sqrt{\lambda_{n}^\star}}}\phi^\star_n(x)\,.
$$
Let $q\geq p\geq 2$ and $s\geq 0$.
Recall here the  Khintchine inequality (see e.g. \cite[Lemma 3.1]{BT2} for a  proof): there exists $C>0$ such that for all real $k\geq 2$ and $(a_{n})\in \ell^{2}(\N)$
\begin{equation}\label{khin} 
\big\|\sum_{n\geq 0}g_{n}(\om)\,a_{n}\big\|_{L_{\bf p}^{k}}\leq C\sqrt{k}\Big(\sum _{n\geq 0}|a_{n}|^{2}\Big)^{\frac12},
\end{equation}
if the $g_n$ are iid normalized Gaussians.
Applying it to \eqref{khin} we get
\begin{equation*}
\|H^{s/2}F(\omega,t,x)\|_{L^q_\omega}\leq C\sqrt{q}\Big(\sum_{n=0}^{\infty}\chi^{2}
\bigl(\frac{\lambda_{n}^\star}{\lambda_{N}^\star}\bigr){\frac{ |\phi^\star_{n}(x)|^2}{\lambda^{\star\; 1-s}_{n}}}\Big)^{1/2}\leq C\sqrt{q}\Big(\sum_{n=0}^{\infty} 
 {\frac{ |\phi^\star_{n}(x)|^2}{\langle n\rangle^{1-s}}}\Big)^{1/2},
\end{equation*}
and using twice the Minkowski inequality for $q\geq p$ gives
\begin{equation}\label{born}
\|H^{s/2}F(\omega,t,x)\|_{L^q_{\omega}L^p_{t,x}}\leq \|H^{s/2}F(\omega,t,x)\|_{L^p_{t,x}L^q_\omega}\leq C\sqrt{q}\Big(\sum_{n=0}^{\infty} 
 {\frac{ \|\phi^\star_{n}(x)\|_{L^{p}(\R^{2})}^2}{\langle n\rangle^{1-s}}}\Big)^{1/2}.
\end{equation}

We are now ready to prove  \eqref{b1}. Set $p=4$ and $s=0$. Since  by Lemma \ref{lem.born} we have $\|\phi^\star_{n}\|_{L^{4}(\R^{2})}\leq C n^{-1/8}$, we get from \eqref{born} 
$$
\|F(\omega,t,x)\|_{L^q_{\omega}L^4_{t,x}}\leq C\sqrt{q}\,.
$$
The Bienaym\'e-Tchebichev inequality gives then
$$
{\bf p}\left( \|F(\omega,t,x)\|_{L^4_{t,x}}>\lambda\,\right)\leq 
(\lambda^{-1}\|F(\omega,t,x)\|_{L^q_{\omega}L^4_{t,x}})^q
\leq
(C\lambda^{-1}\sqrt{q})^{q}\,.
$$
Thus by choosing $q=\delta \lambda^2\geq 4$, for $\delta$ small enough, we get the bound 
$$
{\bf p}\big(\|F(\omega,t,x)\|_{L^4_{t,x}}>\lambda\,\big)\leq Ce^{-c\lambda^2}\,,
$$
which is \eqref{b1}. 

For the proof of \eqref{b2},  we analyze the function
$$
G(\omega,t,x)\equiv\sum_{n=0}^{\infty} e^{-it\lambda_{n}^\star}\left( \chi(\frac{\lambda_{n}^\star}{\lambda_{N}^\star})-
\chi(\frac{\lambda_{n}^\star}{\lambda_{N_{0}}^\star}) \right)\frac{g_{n}(\omega)}{\sqrt{\lambda_{n}^\star}} \phi^\star_n(x),
$$
and we use that a negative power of $N_0$ can be gained in the estimate. 
Namely, there is $\gamma>0$ such that
$$
\|F_{N_0}(\omega,t,x)\|_{L^q_{\omega}L^4_{t,x}}\leq C\sqrt{q}N_0^{-\gamma}\,,
$$
which implies \eqref{b2}. 

To prove \eqref{b3}-\eqref{b4}, we come back to \eqref{born} and argue similarly.
This completes the proof of Lemma~\ref{ld}.
\end{proof}

   \begin{lemm}
 Let $\beta\geq 0$. Let  $p\in [1,\infty[$, then when $N\longrightarrow +\infty$. 
\begin{equation*} 
C^{N}_{\b}\e^{-\beta \H_{N}(u)}\longrightarrow C_{\b}\e^{-\beta \H(u)}\quad \text{in}\quad L^{p}(\text{d}\mu_\star(u)).
\end{equation*}

 In particular,  for all measurable sets $A\subset X^{0}_{\star}(\R^{2})$, 
 \begin{equation*}
\rho^{N}_{\star}(A)\longrightarrow \rho_{\star}(A).
 \end{equation*}
  \end{lemm}

\begin{proof} Denote by $ G^{N}_{\beta}(u)= \e^{-\beta \H_{N}(u)}$ and $ G_{\beta}(u)= \e^{-\beta \H(u)}$. By \eqref{b2}, we deduce that $ \H_{N}(u)  \longrightarrow\H(u)$ in measure, w.r.t. $\mu_\star$. In other words, for  $\eps>0$ and $N\geq 1$ we denote by
\begin{equation*}
A_{N,\eps}=\big\{  \,u\in X^{0}_{\star}(\R^{2})\;:\;|G^{N}_{\beta}(u)-G_{\beta}(u)| \leq\eps  \},
\end{equation*}
then $\mu_{\star}({A^{c}_{N,\eps}})\longrightarrow 0$, when $N\longrightarrow +\infty$. Since $0\leq G,G_{N}\leq 1$,
\begin{eqnarray*}
\|G_{\beta}-G^{N}_{\beta}\|_{L^{p}_{\mu_{\star}}}&\leq &\|(G_{\beta}-G^{N}_{\beta}){\bf 1}_{A_{N,\eps}}\|_{L^{p}_{\mu_{\star}}}+\|(G_{\beta}-G^{N}_{\beta}){\bf 1}_{{A^{c}_{N,\eps}}}\,\|_{L^{p}_{\mu_{\star}}}\\
&\leq & \eps \big(\,\mu_{\star}(A_{N,\eps}\,)\big)^{1/p}+2\big(\,\mu_{\star}({A^{c}_{N,\eps}})\,\big)^{1/p}\leq C\eps,
\end{eqnarray*}
for $N$ large enough. Finally, we have when $N\longrightarrow +\infty$
$$C^{N}_{\beta}=\big(\int\e^{-\beta \H_{N}(u)}d\mu_{\star}(u)\big)^{-1}\longrightarrow \big(\int\e^{-\beta \H(u)}d\mu_{\star}(u)\big)^{-1} =C_{\beta} ,$$
and this ends the proof.
\end{proof}
  
We look for a solution to~\eqref{CR} of the form $u=f+v$, thus $v$ has to satisfy
\begin{equation}\label{eqt1} 
\left\{
\begin{aligned}
&i\partial_{t}v=\T(f+v), \quad   (t,x)\in \R\times \R^2,\\
&v(0,x)=0,
\end{aligned}
\right.
\end{equation}
 with $\T(u)=\T(u,u,u)$. Similarly, we introduce   
    \begin{equation}\label{eqt2} 
\left\{
\begin{aligned}
&i\partial_{t}v_N=\T_N(f+v_N), \quad   (t,x)\in \R\times \R^2,\\
&v(0,x)=0.
\end{aligned}
\right.
\end{equation}

Recall that $X^{0}_{\star}(\R^2)=X^{0}_{hol}(\R^2)$ or $X^{0}_{rad}(\R^2)$. Define the sets, for $s<\frac{1}{2}$, 
\begin{equation*}
A_{rad}^s(D)=\big\{f\in X^{0}_{rad}(\R^2)\,:\;\; \Vert \e^{-itH}f\big\Vert_{L_{[-\frac{\pi}4,\frac{\pi}4]}^4\W^{s,4}(\R^2)}\leq D \big\},
\end{equation*}
and choosing $p(s)=\frac{4}{1-2s}$ so that $s< \frac12-\frac1p$,
\begin{equation*}
A_{hol}^s(D)=\big\{f\in X^{0}_{hol}(\R^2)\,:\;\;  \| {\e^{-it H}f}\|_{L^{8/3}_{ [-\frac{\pi}4,\frac{\pi}4] } L^{8/3}(\R^{2})}+\Vert \e^{-itH}f\big\Vert_{L_{[-\frac{\pi}4,\frac{\pi}4]}^{p(s)}\W^{s,p(s)}(\R^2)}\leq D \big\}.
\end{equation*}

In the sequel we write $A_{\star}(D)=A_{hol}(D)$ or $A_{rad}(D)$.
Then we have the following result

 \begin{lemm}
Let $\beta \geq 0$. There exist $c,C>0$ so that for all $N\geq 0$
\begin{equation*}
\rho^{N}_{\star}\big(  A_{\star}(D)^{c}  \big)\leq C\e^{-cD^{2}}, \quad \rho_{\star}\big(  A_{\star}(D)^{c}  \big)\leq C\e^{-cD^{2}},\quad \mu_{\star}\big(  A_{\star}(D)^{c}  \big)\leq C\e^{-cD^{2}}.
\end{equation*}
  \end{lemm}

\begin{proof}
Since $\beta\geq 0$, we have $\rho^{N}_{\star}\big(  A_{\star}(D)^{c}  \big), \rho_{\star}\big(  A_{\star}(D)^{c}  \big) \leq C \mu_{\star}\big(  A_{\star}(D)^{c}  \big)$. The result is therefore given by  \eqref{b3} and \eqref{b4}.
\end{proof}

\begin{prop}\label{proplocal}
Let $s<1/2$. There exists $c>0$ such that, for any $D\geq 0$, setting ${\tau(D) =cD^{-2}}$, for any $f\in A(D)$ there exists a  unique solution $v\in L^{\infty}([-\tau,\tau]; L^2(\R^2))$ to the equation\;\eqref{eqt1} and a  unique solution $v_N\in L^{\infty}([-\tau,\tau]; L^2(\R^2))$ to the equation\;\eqref{eqt2} which furthermore satisfy 
\begin{equation*}
\Vert v \big\Vert_{L^{\infty}([-\tau,\tau]; \H^s(\R^2) )},\quad \Vert v_N \big\Vert_{L^{\infty}([-\tau,\tau]; \H^s(\R^2) )}\leq  D.
\end{equation*}
\end{prop}

The key ingredient in the proof of this result, is the following trilinear estimate

  \begin{lemm} \ph \label{lemTri2} Assume that for $1 \leq j \leq 3$ and $1\leq k\leq 4$, $(p_{jk},q_{jk})\in [2,+\infty[^2$ are Strichartz admissible pairs, or in other words satisfy
\begin{equation*}
\frac1{q_{jk}}+\frac{1}{p_{jk}}=\frac{1}2,
\end{equation*}
and are such that for $1\leq j\leq 4$,
  \begin{equation*}
\frac{1}{p_{j1}}+  \frac{1}{p_{j2}}+\frac{1}{p_{j3}}+  \frac{1}{p_{j4}}=\frac{1}{q_{j1}}+  \frac{1}{q_{j2}}+\frac{1}{q_{j3}}+  \frac{1}{q_{j4}}=1.
  \end{equation*}
  Then for all $s\geq 0$,  there exists $C>0$ such that 
\begin{align*}
\Vert \T(u_1,u_2,u_3)\Vert_{\H^s(\R^2)}\leq &
C  \big\Vert \e^{-itH}u_1\big\Vert_{L^{p_{11}}\W^{s,q_{11}}} \big\Vert \e^{-itH}u_2\big\Vert_{L^{p_{12}}L^{q_{12}}} \big\Vert \e^{-itH}u_3\big\Vert_{L^{p_{13}}L^{q_{13}}}\\
& \quad +C  \big\Vert \e^{-itH}u_1\big\Vert_{L^{p_{21}}L^{q_{21}}} \big\Vert \e^{-itH}u_2\big\Vert_{L^{p_{22}}\W^{s,q_{22}}} \big\Vert \e^{-itH}u_3\big\Vert_{L^{p_{23}}L^{q_{23}}}\\
& \quad +C  \big\Vert \e^{-itH}u_1\big\Vert_{L^{p_{31}}L^{q_{31}}} \big\Vert \e^{-itH}u_2\big\Vert_{L^{p_{32}}L^{q_{32}}} \big\Vert \e^{-itH}u_3\big\Vert_{L^{p_{33}}\W^{s,q_{33}}},
\end{align*}
with the notation $L^p\W^{s,q}=L^p\big([-\frac{\pi}4,\frac{\pi}4 ];\W^{s,q}(\R^2)\big)$.
  \end{lemm}

\begin{proof}
 By duality
 \begin{eqnarray*}
 \Vert \T(u_1,u_2,u_3)\Vert_{\H^s(\R^2)}&=& \sup_{\|u\|_{ L^{2}(\R^{2})}=1}\<  H^{s/2} \T(u_1,u_2,u_3),u   \>_{L^{2}(\R^{2})}\\
 &=& 2\pi   \sup_{\|u\|_{ L^{2}(\R^{2})}=1}\int_{-\frac{\pi}4}^{\frac{\pi}4}\int_{\R^2} H^{s/2}\Big( (\e^{-it H}u_1)   (\e^{-it H}u_2) (\ov{ \e^{-it H}u_3}) \Big)(\ov{\e^{-it H}u}) dx\, dt .
 \end{eqnarray*}
Then by Strichartz for all $u$ of unit norm in $L^2$ and  for any admissible pair $(p_{4},q_{4})$
  \begin{eqnarray*}
 \Vert \T(u_1,u_2,u_3)\Vert_{\H^s(\R^2)}&\leq &C \|      (\e^{-it H}u_1)   (\e^{-it H}u_2) (\ov{ \e^{-it H}u_3})\|_{L^{p'_{4}}\W^{s,q'_{4}}}   \|      \e^{-it H}u  \|_{L^{p_{4}}L^{q_{4}}}\\
 &\leq & C \|      (\e^{-it H}u_1)   (\e^{-it H}u_2) (\ov{ \e^{-it H}u_3})\|_{L^{p'_{4}}\W^{s,q'_{4}}}.
 \end{eqnarray*}
 We then conclude using \eqref{equiv} and  applying twice  the following lemma.
 \end{proof}

We have the following product rule

\begin{lemm}\label{lemmprod2}
 Let $s\geq 0$, then the following estimates hold
\begin{equation*} 
\|u\,v\|_{\W^{s,q}}\leq C \|u\|_{L^{q_{1}}}\|v\|_{\W^{s,{q'_{1}}}}+C \|v\|_{L^{q_{2}}}\|u\|_{\W^{s,{q'_{2}}}},
\end{equation*}
~\\[-5pt]
with $1<q<\infty$, $1< q_{1},\,q_{2}<   \infty$ and  $1\leq  {q'_{1}},\,{q'_{2}}<  \infty$  so that 
$$\frac1q=\frac1{q_{1}}+\frac1{{q'_{1}}}=\frac1{q_{2}}+\frac1{{q'_{2}}}.$$
\end{lemm}
For the proof with the usual Sobolev spaces, we refer to    \cite[Proposition 1.1, p. 105]{Taylor}. The result in our context follows by using \eqref{equiv}.

    \begin{proof}[Proof of Proposition \ref{proplocal}]
We only consider the equation \eqref{eqt1}, the other case being similar by the boundedness of $S_N$ on $L^p(\R^2)$. For $s<1/2$, we define the space 
\begin{equation*}
Z^s(\tau)=\big\{v\in \mathcal{C}\big([-\tau,\tau]; \H^s(\R^{2})\big)\;\,\text{s.t.}\;\; v(0)=0 \;\text{ and }\; \|v\|_{Z^s(\tau)}\leq D\big\},
\end{equation*}
with $\|v\|_{Z^s(\tau)}=\|v\|_{L^{\infty}_{[-\tau,\tau]}\H^{s}(\R^{2})}$,  and for $f\in A_{\star}(D)$ we define the operator
\begin{equation*}
K(v)=-i \int_{0}^{t}\T(f+v)ds.
\end{equation*} 
We will show that $K$ has a unique fixed point  $v\in Z^s(\tau)$.

\bigskip

\noindent
\underline{The case of radial Hermite functions.}  By Lemma \ref{lemTri2} with $(p_{jk},q_{jk})=(4,4)$, we have for all $v\in Z^s(\tau)$
\begin{eqnarray}\label{kk}
\|K(v)\|_{Z^s(\tau)}&\leq & \tau \big\|\T(f+v)\big\|_{Z^s(\tau)}\nonumber\\
&\leq & C \tau \Big\| \| \e^{-is H}(f+v)(t)\|^{3}_{L^{4}_{s\in [-\frac{\pi}4,\frac{\pi}4] } \W^{s,4}(\R^{2})}\Big\|_{L^{\infty}_{t \in [-\tau,\tau]}}.
\end{eqnarray}
Next, by Strichartz and since $v \in Z^s(\tau)$
\begin{eqnarray*}
\big \| \e^{-is H}(f+v)(t)\big\|_{L^{4}_{s\in [-\frac{\pi}4,\frac{\pi}4] } \W^{s,4}(\R^{2})}&\leq&  \| \e^{-is H}f\|_{L^{4}_{s\in [-\frac{\pi}4,\frac{\pi}4] } \W^{s,4}(\R^{2})}+ \| \e^{-is H}v(t)\|_{L^{4}_{s\in [-\frac{\pi}4,\frac{\pi}4] } \W^{s,4}(\R^{2})}\\
 &\leq& C \big( D + \| v(t)\|_{ \H^{s}({\R^{2}})} \big) \leq 2CD.
\end{eqnarray*}
Therefore, from \eqref{kk} we deduce 
\begin{equation*}
\|K(v)\|_{Z^s(\tau)}\leq C \tau D^{3},
\end{equation*}
which implies that $K$ maps $Z^s(\tau)$ into itself when $\tau \leq cD^{-2}$, for $c>0$ small enough.

Similarly, for $v_{1},v_{2}\in Z^s(\tau)$, we have the bound
\begin{equation}\label{ck}
\|K(v_{2})-K(v_{1})\|_{Z^s(\tau)}\leq C \tau D^{2}\|v_{2}-v_{1}\|_{Z^s(\tau)},
\end{equation}
which shows that if $\tau \leq cD^{-2}$, $K$ is a contraction of $Z^s(\tau)$. The Picard fixed point theorem gives the desired result.

\bigskip

\noindent
\underline{The case of holomorphic Hermite functions.} For $s<\frac 12$, recall that we set $p=p(s)=\frac{4}{1-2s}$ so that $s< \frac12-\frac1p$.
\begin{eqnarray*} 
\|K(v)\|_{Z^s(\tau)}&\leq & \tau \big\|\T(f+v)\big\|_{Z^s(\tau)}\nonumber\\
&\leq & C \tau \left( \left\| \mathcal{T}(f,f,f)\right\|_{Z^s} +  \left\| \mathcal{T}(f,f,v)\right\|_{Z^s} +  \left\| \mathcal{T}(f,v,v)\right\|_{Z^s} +  \left\| \mathcal{T}(v,v,v)\right\|_{Z^s} \right).
\end{eqnarray*}
We estimate each term thanks to Lemma \ref{lemTri2} and Strichartz. The conjugation plays no role, so we forget it.\\
\indent Estimate of the trilinear term in $v$:
\begin{equation*}
\left\| \mathcal{T}(v,v,v)\right\|_{\mathcal{H}^s} \leq C  \| {\e^{-it' H}v}\|^3_{L^{4}_{t'\in [-\frac{\pi}4,\frac{\pi}4] } \W^{s,4}(\R^{2})}\leq C  \| v\|^3_{\H^{s}(\R^{2})}.
\end{equation*}
\indent Estimate of the quadratic term in $v$: for $\delta>0$ such that $\frac{2}{\frac{8}{3} + \delta} + \frac 1p + \frac 14 = 1$,
\begin{align*}
\left\| \mathcal{T}(v,v,f)\right\|_{\mathcal{H}^s} & \leq  C  \| {\e^{-it' H}v}\|^2_{L^{8/3+\delta}_{t'\in [-\frac{\pi}4,\frac{\pi}4] } L^{8/3+\delta}(\R^{2})} \| {\e^{-it' H}f}\|_{L^{p}_{t'\in [-\frac{\pi}4,\frac{\pi}4] } \W^{s,p}(\R^{2})}\\
&\qquad  + \| {\e^{-it' H}v}\|_{L^{4}_{t'\in [-\frac{\pi}4,\frac{\pi}4] } \W^{s,4}(\R^{2})}\| {\e^{-it' H}v}\|_{L^{4}_{t'\in [-\frac{\pi}4,\frac{\pi}4] } L^{4}(\R^{2})} \| {\e^{-it' H}f}\|_{L^{4}_{t'\in [-\frac{\pi}4,\frac{\pi}4] } L^{4}(\R^{2})}\\
&\leq C D \| v\|^2_{\H^{s}(\R^{2})}.
\end{align*}

Estimate of the linear term in $v$: with the same $\delta$ as above,
\begin{align*}
\left\| \mathcal{T}(v,f,f)\right\|_{\mathcal{H}^s} & \leq  C  \| {\e^{-it' H}v}\|_{L^{8/3+}_{t'\in [-\frac{\pi}4,\frac{\pi}4] } L^{8/3+}(\R^{2})} \| {\e^{-it' H}f}\|_{L^{8/3+}_{t'\in [-\frac{\pi}4,\frac{\pi}4] } L^{8/3+}(\R^{2})} \| {\e^{-it' H}f}\|_{L^{p}_{t'\in [-\frac{\pi}4,\frac{\pi}4] } \W^{s,p}(\R^{2})}\\
& \qquad + \| {\e^{-it' H}v}\|_{L^{4}_{t'\in [-\frac{\pi}4,\frac{\pi}4] } \W^{s,4}(\R^{2})}  \| {\e^{-it' H}f}\|^2_{L^{4}_{t'\in [-\frac{\pi}4,\frac{\pi}4] } L^{4}(\R^{2})}\\
& \leq C D^2 \| v\|_{\H^{s}(\R^{2})}.
\end{align*}

Estimate of the constant term in $v$:
\begin{align*}
\left\| \mathcal{T}(v,f,f)\right\|_{\mathcal{H}^s}
 \leq C  \| {\e^{-it' H}f}\|^2_{L^{8/3+}_{t'\in [-\frac{\pi}4,\frac{\pi}4] } L^{8/3+}(\R^{2})} \| {\e^{-it' H}f}\|_{L^{p}_{t'\in [-\frac{\pi}4,\frac{\pi}4] } \W^{s,p}(\R^{2})}
\leq C D^3 .
\end{align*}
With these estimates at hand, the result follows by the Picard fixed point theorem.
      \end{proof}
      
\subsubsection{ \bf Approximation and invariance of the measure}

\begin{lemm}
Fix $D\geq 0$ and $s<1/2$. Then for all $\eps>0$, there exists $N_0\geq 0$ such that for all $f\in A_{\star}(D)$ and $N\geq N_0$
\begin{equation*}
\big \Vert \Phi(t)f- \Phi_N(t)f \big\Vert_{L^{\infty}([-\tau_1,\tau_1]; \H^s(\R^2) )}\leq \eps,
\end{equation*}
where  $\tau_1=cD^{-2}$ for some  $c>0$.
\end{lemm}

\begin{proof}
Denoting for simplicity $\mathcal{T}(f) = \mathcal{T}(f,f,f)$,
\begin{equation*}
v-v_N=-i\int_0^{t}\big[S_N\big(\T(f+v)-\T(f+v_N)\big)+(1-S_N)\T(f+v) \big]ds.
\end{equation*}
As in \eqref{ck} we get 
\begin{equation*}
\|v-v_{N}\|_{Z^s(\tau)}\leq C \tau D^{2}\|v-v_{N}\|_{Z^s(\tau)}+\int_{-\tau}^{\tau}\|(1-S_N)\T(f+v) \|_{\H^s(\R^2)}ds ,
\end{equation*}
which in turn implies when $C \tau D^{2}\leq 1/2$
\begin{equation*}
\|v-v_{N}\|_{Z^s(\tau)}\leq 2\int_{-\tau}^{\tau}\|(1-S_N)\T(f+v) \|_{\H^s(\R^2)}ds.
\end{equation*}
Let $\eta>0$ so that $s+\eta<1/2$. Then by the proof of Proposition \ref{proplocal}, $\|\T(f+v) \|_{L_{[-\tau,\tau]}^{\infty}{\H^{s+\eta}(\R^2)}} \leq CD^3$ if $\tau \leq c_0 D^{-2}$ and therefore there exists $N_{0}=N_{0}(\eps,D)$ which satisfies the claim.
\end{proof}

In the next result, we summarize the results obtained by Suzzoni in \cite[Sections 3.3 and 4]{Suzzoni}. Since the proofs are very similar in our context, we skip them.

Let $D_{i,j}=(i+j^{1/2})^{1/2}$, with $i,j\in \N$ and set $T_{i,j}=\sum_{\ell=1}^j \tau_1(D_{i,\ell})$. Let 
\begin{equation*}
\Sigma_{N,i}:=\big\{f\,: \forall j\in \N, \;\; \Phi_N(\pm T_{i,j})f \in A(D_{i,j+1})\big\},
\end{equation*}
 and 
 \begin{equation*}
 \Sigma_i:=\limsup_{N\to +\infty} \Sigma_{N,i}, \qquad \Sigma:=\bigcup_{i\in \N} \Sigma_i.
 \end{equation*}
 
\begin{prop}
Let $\beta\geq 0$, then 
\begin{enumerate}[(i)] 
\item The set $\Sigma$ is of full $\rho_{\star}$ measure.
\item For all $f\in \Sigma$, there exists a unique global solution $u=f+v$ to~\eqref{CR}, and for all $t\in \R$, $u(t)\in \Sigma$. This define a global flow $\Phi$ on $\Sigma$
\item For all measurable set $A\subset \Sigma$, and all $t\in \R$, 
\begin{equation*}
\rho_\star(A)=\rho_\star(\Phi(t)(A)).
\end{equation*}
\end{enumerate}
\end{prop}

 \section{Weak solutions: proof of Theorem  \texorpdfstring{\ref{thm2}  }{Theorem 2.6}}\label{sect4}
 
  \subsection{\bf Definition of \texorpdfstring{$ \boldsymbol \T(u,u,u)$}{\T(u,u,u)}  on the support of  \texorpdfstring{$\mu$}{mu} } For $N\geq 0$, denote by $\Pi_{N}$ the orthogonal projector on the space $\bigoplus_{k=0}^{N} E_{k}$ (in this section, we do not need  the smooth cut-offs $S_{N}$).
 In the sequel, we denote by ${\T(u)=\T(u,u,u)}$ and $\T_N(u)=\Pi_N \T(\Pi_N u,\Pi_N u,\Pi_N u)$
 \begin{prop}\ph\label{Prop.cauchy}
For all $p\geq 2$ and $\s>1$, the sequence $\big(\T_{N}(u)\big)_{N\geq1}$ is  a Cauchy  sequence in $L^{p}\big(X^{-1}(\R^2),\mathcal{B},d\mu; \H^{-\s}(\R^2)\big)$. Namely, for all $p\geq	 2$, there exist $\delta>0$ and $C>0$ so that for all $1\leq M<N$,
 \begin{equation*}
\int_{X^{-1}(\R^2)}\|\T_{N}(u)-\T_{M}(u)\|^{p}_{\H^{-\s}(\R^2)}\text{d}\mu(u)\leq C M^{-\delta}.
\end{equation*}
We denote by $\T(u)=\T(u,u,u)$ the limit of this sequence and we have for all $p\geq 2$  
\begin{equation}\label{Tlp}
\|\T(u) \|_{L^p_{\mu}\H^{-\s}(\R^2)}\leq C_p.
\end{equation}
 \end{prop}

Before we turn to the proof of Proposition \ref{Prop.cauchy}, let us state two elementary results which will be needed in the sequel

\begin{lemm}\ph \label{lemB}
For all $n\in \N$,
  \begin{equation*}
   \sum_{k=n}^{+\infty} \frac1{2^k}\binom{k}{n}=  \sum_{k=n}^{+\infty} \frac{k!\,}{2^{k}\,n !\, (k-n) !\,}=2.
  \end{equation*}
\end{lemm}

\begin{proof}
For $\vert z\vert<1$ we have $\dis \frac1{1-z}=\sum_{k=0}^{+\infty} z^k$. If one differentiates $n$ times this formula we get 
\begin{equation*}
 \frac{n !\,}{(1-z)^{n+1}}=\sum_{k=n}^{+\infty}   \frac{k!\,}{   (k-n) !\,} z^{k-n},
\end{equation*}
which implies the result, taking $z=1/2$.
\end{proof}

\begin{lemm}\ph \label{lemA}
Let $0<\eps<1$ and $p,L\geq 1$ so that $p\leq L^{\eps}$. Then 
  \begin{equation*}
  \frac{L!\,}{   2^{L}  \,   (L-p)   !\, }\leq C  2^{-L/2}.
  \end{equation*}
\end{lemm}

\begin{proof}
The proof is straightforward. By the assumption  $p\leq L^{\eps}$
\begin{equation*}
  \frac{L!\,}{       (L-p)!\, } \leq L^p \leq C 2^{L/2},
   \end{equation*}
   which was the claim.
\end{proof}

\begin{proof}[Proof of Proposition \ref{Prop.cauchy}]
 By the result \cite[Proposition 2.4]{ThTz} on the Wiener chaos, we only have to prove the statement for $p=2$.
 
Firstly, by definition of the measure $\mu$ 
\begin{equation*}
\int_{X^{-1}(\R^2)}\|\T_{N}(u)-\T_{M}(u)\|^{2}_{\H^{-\s}(\R^2)}\text{d}\mu(u)=\int_{\Omega}\|\T_{N}\big(\gamma(\om)\big)-\T_{M}\big(\gamma(\om)\big)\|^{2}_{\H^{-\s}(\R^2)}\text{d}{\bf p}(\om).
\end{equation*}
Therefore, it is enough to prove that  $\big(\T_{N}(\gamma)\big)_{N\geq1}$ is a Cauchy sequence in $L^{2}\big(\Omega; \H^{-\s}(\R^2)\big)$. 
Let $1\leq M<N$ and fix $\alpha>1/2$. 
By \eqref{tfi} we get
\begin{eqnarray*}
H^{-\alpha}\T_{N}(\gamma)&=&\frac1{2^{\alpha}}\sum_{A_N}\frac{g_{n_{1}}{g}_{n_{2}}\ov{g_{n_{3}}}}{   (n_1+n_2-n_3+1)^{\alpha}}\T(\phi^{hol}_{n_{1}},\phi^{hol}_{n_{2}},\phi^{hol}_{n_{3}})\\
&=&\frac{\pi}{8\cdot 2^{\alpha}}\sum_{A_N}\frac{(n_1+n_2)!\,}{  2^{n_1+n_2} \sqrt{n_1!\, n_2!\, n_3!\, (n_1+n_2-n_3)!\,}}  \frac{g_{n_{1}}{g}_{n_{2}}\ov{g_{n_{3}}}}{   (n_1+n_2-n_3+1)^{\alpha}}   \phi^{hol}_{n_1+n_2-n_3}\\
&=&\frac{\pi}{8\cdot 2^{\alpha}}\sum_{p=0}^{N} \frac{1}{   (p+1)^{\alpha}} \Big( \sum_{A^{(p)}_N}   \frac{(n_1+n_2)!\,}{  2^{n_1+n_2} \sqrt{n_1!\, n_2!\, n_3!\, p!\,}} \, {g_{n_{1}}{g}_{n_{2}}\ov{g_{n_{3}}}}\Big)\phi^{hol}_{p},
\end{eqnarray*}
with
\begin{equation*}
A_{N}=\big\{n\in \N^{3} \;\;\text{s.t.}\;\; 0\leq {n_j}\leq N, 
\;\;0\leq {n_1}+{n_2}-{n_3}\leq N\big\},
\end{equation*} 
\begin{equation*}
A^{(p)}_{N}=\big\{n\in \N^{3} \;\;\text{s.t.}\;\; 0\leq {n_j}\leq N, 
\;\; {n_1}+{n_2}-{n_3}=p\big\} \quad \mbox{if \;$0\leq p \leq N$}.
\end{equation*} 
 Therefore,
 \begin{multline*}
 \Vert  \T_{N}(\gamma)- \T_{M}(\gamma) \Vert^2_{\H^{-\a}(\R^2)}=\\
 =\frac{\pi^2}{64\cdot 2^{2\alpha}} \sum_{ p= 0  }^{N} \frac{1}{(p+1)^{2\alpha}}\sum_{(n,m) \in A^{(p)}_{M,N}\times  A^{(p)}_{M,N}}\frac{(n_1+n_2)!\, (m_1+m_2)!\, g_{n_{1}}{g}_{n_{2}}\ov{g_{n_{3}}}  \ov{g_{m_{1}}{g}_{m_{2}}}{g_{m_{3}}}}{   2^{n_1+n_2} 2^{m_1+m_2}p!\, \sqrt{n_1!\, n_2!\, n_3!\, } \sqrt{m_1!\, m_2!\, m_3!\, } }   \end{multline*}
 where $A^{(p)}_{M,N}$ is the set defined by
\begin{multline*}
A^{(p)}_{M,N}=\Big\{n\in \N^{3} \;\;\text{s.t.}\;\; 0\leq  {n_j}\leq N,
\;\;{n_1}+ {n_2}- {n_3} = p \in \{0 \dots N\},  \\
\;\;    \;\;\big( {n_1}>M        \;\;\text{or}\;\;   {n_2}>M     \;\;\text{or}\;\;   {n_3}>M \;\;\text{or}\;\; p>M \big) \Big\}.
\end{multline*} 
 Now we take the integral over $\Omega$. Since $(g_{n})_{n\geq 0}$ are independent and centred Gaussians,  we deduce that each term in the r.h.s. vanishes, unless \\[5pt]
 \indent  {\bf Case 1:}
 $(n_{1},n_2,n_3 )=(m_{1},m_2,m_3 )$ or $(n_{1},n_2,n_3 )=(m_{2},m_1,m_3)$ \\[5pt]
 or 
 \\[5pt]
 \indent  {\bf Case 2:}
$(n_{1},n_2,  m_1)=(n_{3},m_2,m_3)$ or  $(n_{1},n_2,  m_2 )=(n_{3},m_1,m_3 )$ or $(n_1,n_2,m_3) = (m_1,n_3,m_2)$ or $(n_1,n_2,m_3) = (m_2,n_3,m_1)$.  \\[5pt]
~
We write
\begin{equation*}
\int_{\Omega} \Vert  \T_{N}(\gamma)- \T_{M}(\gamma) \Vert^2_{\H^{-2\a}(\R^2)} d{\bf p}=J_1+J_2,
\end{equation*}
where $J_1$ and $J_2$ correspond to the contribution in the sum  of each of the previous cases.\ligne

 {\bf Contribution in case 1:} By symmetry, we can assume that $(n_{1},n_2,n_3)=(m_{1},m_2,m_3)$. Define 
\begin{equation*}
B_{M,N}^{(p)}=  \Big\{n\in \N^{2}\;\;\text{s.t.}\;\; 0\leq  {n_j}\leq N,  \;\; \text{and}
     \;\;\big( {n_1}>M        \;\;\text{or}\;\;   {n_2}>M     \;\;\text{or}\;\; {n_1+n_2-p}>M  \;\;\text{or}\;\; p>M \big)    \Big\}.
\end{equation*} 
Then 
\begin{equation*}
J_1\leq C\sum_{ p\geq 0  } \frac{1}{(1+p)^{2\alpha}}\sum_{B_{M,N}^{(p)}}\frac{\big((n_1+n_2)!\,\big)^2 }{   2^{2(n_1+n_2)}  p!\, n_1!\, n_2!\, (n_1+n_2-p)!\, } .
\end{equation*}
In the previous sum, we make the change of variables $L=n_1+n_2$, and we observe that on $B_{M,N}^{(p)}$, we have $L\geq M$, then
\begin{eqnarray*}
J_1&\leq &C\sum_{p\geq 0 }  \frac{1}{(1+p)^{2\alpha}}\ \sum_{L\geq p+M} \sum_{n_1=0}^{L} \frac{(L!\,)^2 }{   2^{2L}  p!\, n_1!\, (L-n_1)!\, (L-p)!\, } \nonumber \\
& = &C\sum_{p\geq 0 }  \frac{1}{(1+p)^{2\alpha}}\ \sum_{L\geq p+M}  \frac{L!\,}{   2^{L}  p!\,   (L-p)!\, }, \nonumber 
\end{eqnarray*}
where we used the fact that $\sum_{n_1=0}^{L}  \binom{L}{n_1}=2^L$. Let $\eps>0$ and split the previous sum into two pieces
  \begin{eqnarray*}
  J_1&\leq& C\sum_{p=0}^{M^{\eps}}  \frac{1}{(1+p)^{2\alpha}}\ \sum_{L=M}^{+\infty}  \frac{L!\,}{   2^{L}  p!\,   (L-p)!\, }
  +C\sum_{p =M^{\eps}+1 }^{+\infty}  \frac{1}{(1+p)^{2\alpha}}\ \sum_{L= p}^{+\infty}  \frac{L!\,}{   2^{L}  p!\,   (L-p)!\, }\\
  &\leq& C\sum_{p=0}^{M^{\eps}}  \frac{1}{(1+p)^{2\alpha}}\ \sum_{L=M}^{+\infty}  \frac{L!\,}{   2^{L}  p!\,   (L-p)!\, }
  +2C\sum_{p =M^{\eps}+1 }^{+\infty}  \frac{1}{(1+p)^{2\alpha}} \\
  &:=&J_{11}+J_{12},
 \end{eqnarray*}
by Lemma \ref{lemB}. For the first sum, we can use Lemma \ref{lemA}, since $p\leq M^{\eps}\leq L^{\eps}$, thus 
\begin{equation*}
J_{11}\leq C\sum_{p=0}^{M^{\eps}}  \frac{1}{(1+p)^{2\alpha}\, p!}\ \sum_{L=M}^{+\infty}  \frac{1}{   2^{L/2}   }\leq C \sum_{L=M}^{+\infty}  \frac{1}{   2^{L/2}   }\leq C M^{-\delta}.
\end{equation*}
Next, clearly $J_{12} \leq C M^{-\delta}$ because $\a>1/2$, and this gives $J_{1} \leq C M^{-\delta}$.
\ligne

  {\bf Contribution in case 2:} We can assume that  $(n_{1},n_2,  m_1)=(n_{3},m_2,m_3)$.  Then for $n,m\in A^{(p)}_{M,N}$ we have $n_2=m_2=p$. Moreover, by symmetry, we can assume that $n_1>M$ or $p>M$. Thus
  \begin{multline*}
  J_2\leq C\sum_{ p\geq 0  } \frac{1}{(1+p)^{2\alpha}}   \sum_{ n_1=M+1  }^{+\infty}  \sum_{ m_1=0 }^{+\infty} \frac{(n_1+p)!\, (m_1+p)!\, }{   2^{n_1+p} 2^{m_1+p} {n_1!\,  }  {m_1!\,  (p!\,)^2 } }\\+C\sum_{ p\geq M+1  } \frac{1}{(1+p)^{2\alpha}}   \sum_{ n_1=0 }^{+\infty}  \sum_{ m_1=0 }^{+\infty} \frac{(n_1+p)!\, (m_1+p)!\, }{   2^{n_1+p} 2^{m_1+p} {n_1!\,  }  {m_1!\,  (p!\,)^2 } }:=J_{21}+J_{22}. 
 \end{multline*}
 To begin with, by Lemma \ref{lemB}, we have 
   \begin{eqnarray*}
  J_{22}&=& C\sum_{ p\geq M+1  } \frac{1}{(1+p)^{2\alpha}}   \Big(   \sum_{ n_1=0  }^{+\infty}  \frac{(n_1+p)!\,  }{   2^{n_1+p}   {n_1!\,  }    p!\,  }  \Big)  \Big( \sum_{ m_1=0 }^{+\infty} \frac{  (m_1+p)!\, }{   2^{m_1+p} {  }  {m_1!\,  p!\, } }\Big) \\
&=& 4C\sum_{ p\geq M+1  } \frac{1}{(1+p)^{2\alpha}} \leq c M^{-\delta}.
 \end{eqnarray*}
Then by Lemma \ref{lemB} again
 \begin{eqnarray*}
  J_{21}&=& C\sum_{ p\geq 0  } \frac{1}{(1+p)^{2\alpha}}   \Big(   \sum_{ n_1=M+1  }^{+\infty}  \frac{(n_1+p)!\,  }{   2^{n_1+p}   {n_1!\,  }    p!\,  }  \Big)  \Big( \sum_{ m_1=0 }^{+\infty} \frac{  (m_1+p)!\, }{   2^{m_1+p} {  }  {m_1!\,  p!\, } }\Big) \\
  &= & 2C\sum_{ p\geq 0  } \frac{1}{(1+p)^{2\alpha}}   \Big(   \sum_{ n_1=M+1  }^{+\infty}  \frac{(n_1+p)!\,  }{   2^{n_1+p}   {n_1!\,  }    p!\,  }  \Big)  \\
   &= & 2C\sum_{ p=0  }^{M^{\eps}} \frac{1}{(1+p)^{2\alpha}}   \Big(   \sum_{ n_1=M+1  }^{+\infty}  \frac{(n_1+p)!\,  }{   2^{n_1+p}   {n_1!\,  }    p!\,  }  \Big) +2C\sum_{ p=M^{\eps}+1  }^{+\infty} \frac{1}{(1+p)^{2\alpha}}   \Big(   \sum_{ n_1=M+1  }^{+\infty}  \frac{(n_1+p)!\,  }{   2^{n_1+p}   {n_1!\,  }    p!\,  }  \Big) \\
    &:= &K_{1}+K_{2}.
 \end{eqnarray*}
 One the one hand, by Lemma \ref{lemA}
 \begin{equation*}
 K_{1}\leq C \Big(\sum_{ p=0  }^{M^{\eps}} \frac{1}{(1+p)^{2\alpha}\, p!}    \Big) \Big(   \sum_{ n_1=M+1  }^{+\infty}   2^{-n_1/2}   \Big)\leq C M^{-\delta},
 \end{equation*}
and one the other hand, by Lemma \ref{lemB}, since $\a >1/2$
  \begin{equation*}
 K_{2}\leq C\sum_{ p=M^{\eps}+1  }^{+\infty} \frac{1}{(1+p)^{2\alpha}} \leq C M^{-\delta}.
 \end{equation*}
 Putting all the estimates together, we get $J_{2}\leq C M^{-\delta}$, which concludes the proof. \end{proof}

\subsection{\bf Study  of the measure \texorpdfstring{$\boldsymbol \nu_{N}$}{nu}}

Let $N\geq 1$. We then consider the following approximation of~\eqref{CR}
\begin{equation}\label{ODE}
\left\{
\begin{aligned}
&i\partial_{t}u  = \T_N(u),\;\;
(t,x)\in\R\times \R^{2},\\
&u(0,x)=  f(x) \in X^{-1}(\R^{2}).
\end{aligned}
\right.
\end{equation}
 
 The equation \eqref{ODE} is an ODE in the frequencies less than $N$,    ${(1-\Pi_N)u(t)=(1-\Pi_N)f}$ and for all $t\in \R$.
 
 The main motivation to introduce this system is the following proposition, whose proof we omit.

\begin{prop} \label{4.4}
The equation \eqref{ODE} has a global flow $\Phi_{N}$. Moreover, the measure $\mu$ is invariant under  $\Phi_{N}$  : For any Borel set $A\subset X^{-1}(\R^{2})$ and for all $t\in \R$, $\mu\big(\Phi_{N}(t)(A)\big)=\mu\big( A\big)$.
\end{prop}
In particular if $\mathscr{L}_{X^{-1}}(v)=\mu$ then for all $t\in \R$, $\mathscr{L}_{X^{-1}}(\Phi_{N}(t)v)=\mu$.

We denote by $\nu_{N}$ the measure on $\mathcal{C}\big([-T,T]; X^{-1}(\R^{2})\big)$, defined as the image measure of $\mu$ by the map 
\begin{equation*}
 \begin{array}{rcc}
X^{-1}(\R^{2})&\longrightarrow& \mathcal{C}\big([-T,T]; X^{-1}(\R^{2})\big)\\[3pt]
\dis  v&\longmapsto &\dis \Phi_{N}(t)(v).
 \end{array}
 \end{equation*}



\begin{lemm}\label{penguin}\ph Let $\s>1$ and  $p\geq 2$. Then there exists $C>0$ so that for all $N\geq1 $
\begin{equation*} 
\big\|\| u\|_{W^{1,p}_{T}\H^{-\s}_{x}}\big\|_{L^{p}_{\nu_{N}}}\leq  C.
\end{equation*}
\end{lemm}  
  
\begin{proof}

Firstly, we have that  for  $\s>1$, $p\geq 2$ and  $N\geq1$
\begin{equation*} 
\big\| \| u\|_{L^{p}_{T}\H^{-\s}_{x}}\big\|_{L^{p}_{\nu_{N}}}\leq  C.
\end{equation*}
Indeed, by the definition of $\nu_N$ and the invariance of $\mu$ by $\Phi_N$ we have
\begin{equation*}
\|  u\|_{L^{p}_{\nu_{N}}L^{p}_{T}\H^{-\s}_{x}}= (2T)^{1/p}\|  v\|_{L^{p}_{\mu}\H^{-\s}_{x}}=(2T)^{1/p} \|\gamma\|_{L^{p}_{\p}\H^{-\s}_{x}}.
\end{equation*}
Then, by the Khintchine inequality \eqref{khin} and \eqref{est.phi}, for all $p\geq 2$
\begin{equation*}
 \|\gamma\|_{L^{p}_{\p}\H^{-\s}_{x}}\leq C\sqrt{p} \|\gamma\|_{L^{2}_{\p}\H^{-\s}_{x}}\leq C.
\end{equation*}
 We refer to \cite[Proposition 3.1]{BTT2} for the details.

Next, we  show that $\dis \big\|  \| \partial_{t}u\|_{L^{p}_{T}\H^{-\s}_{x}}\big\|_{L^{p}_{\nu_{N}}}\leq  C. $ By definition of  $\nu_N$
\begin{eqnarray*}
\|\partial_{t}u\|^{p}_{L^{p}_{\nu_{N}}L^{p}_{T}\H^{-\s}_{x}}&=& \int_{\mathcal{C}\big([-T,T]; X^{-1}(\R^{2})\big)}\|\partial_{t} u\|^{p}_{L^{p}_{T}\H^{-\s}_{x}} \text{d}\nu_{N}(u)\\
&=& \int_{X^{-1}(\R^{2})}\|\partial_{t} \Phi_{N}(t)(v)\|^{p}_{L^{p}_{T}\H^{-\s}_{x}} \text{d}\mu(v).
\end{eqnarray*}
Now, since $\Phi_{N}(t)(v)$ satisfies \eqref{ODE} and by the invariance of $\mu$, we have 
\begin{eqnarray*}
\|\partial_{t}u\|^{p}_{L^{p}_{\nu_{N}}L^{p}_{T}\H^{-\s}_{x}}&=&  \int_{X^{-1}(\R^{2})}\|\T_N(\Phi_{N}(t)(v))\|^{p}_{L^{p}_{T}\H^{-\s}_{x}} \text{d}\mu(v)\\
&=&  2T \int_{X^{-1}(\R^{2})}\|   \T_N(v)\|^{p}_{\H^{-\s}_{x}} \text{d}\mu(v),
\end{eqnarray*}
and  conclude with \eqref{Tlp} and Proposition \ref{Prop.cauchy}.
\end{proof}

\subsection{\bf The convergence argument}~
The importance of Proposition \ref{penguin} above comes from the fact that it allows to establish the following tightness result for the measures $\nu_N$. We refer to \cite[Proposition~4.11]{BTT2} for the proof. 
\begin{prop}\ph\label{Prop.tight}
Let $T>0$ and $\s>1$. Then  the family of measures 
$$ (\nu_{N})_{N \geq 1} \quad \mbox{with} \quad \nu_N =\mathscr{L}_{\mathcal{C}_{T}\H^{-\s}}\big(u_{N}(t);t\in [-T,T]\big)$$
 is tight in $\mathcal{C}\big([-T,T]; \H^{-\s}(\R^2)\big)$.
\end{prop}

The result of Proposition \ref{Prop.tight} enables us to use the Prokhorov theorem: For each $T>0$ there exists a sub-sequence $\nu_{N_{k}}$ and a measure $\nu$ on the space $\mathcal{C}\big([-T,T]; X^{-1}(\R^{2})\big)$ so that for all $\tau>1 $ and  all bounded continuous function $F: \mathcal{C}\big([-T,T]; \H^{-\tau}(\R^{2}) \big)\longrightarrow \R$
$$
\int_{\mathcal{C}\big([-T,T]; \H^{-\tau}(\R^{2})\big) }F(u)\text{d}\nu_{N_{k}}(u)\longrightarrow  \int_{\mathcal{C}\big([-T,T]; \H^{-\tau}(\R^{2})\big) }F(u)\text{d}\nu(u).
$$
By the Skohorod theorem, there  exists a probability space $(\widetilde{\Omega},\widetilde{\mathcal{F}},\widetilde{\bf p})$, a sequence of random variables~$(\widetilde{u}_{N_{k}})$ and a random variable~$\widetilde{u}$ with values in $\mathcal{C}\big([-T,T]; X^{-1}(\R^{2})\big)$ so that 
\begin{equation}\label{loi}
\mathscr{L}\big(\widetilde{u}_{N_{k}};t\in [-T,T]\big)=\mathscr{L}\big(u_{N_{k}};t\in [-T,T]\big)=\nu_{N_{k}}, \quad \mathscr{L}\big(\widetilde{u};t\in [-T,T]\big)=\nu,
\end{equation}
 and for all $\tau >1$
\begin{equation}\label{CV}
\widetilde{u}_{N_{k}}\longrightarrow \widetilde{u},\quad \;\;\widetilde{\bf p}-\text{a.s. in}\;\; \mathcal{C}\big([-T,T]; \H^{-\tau}(\R^{2})\big).
\end{equation}

We now claim that $\mathscr{L}_{X^{-1}}( {u}_{N_{k}}(t))=\mathscr{L}_{X^{-1}}(\widetilde{u}_{N_{k}}(t))=\mu$, for all~$t\in [-T,T]$ and $k\geq 1$. Indeed, for all~$t\in [-T,T]$, the evaluation map 
\begin{equation*}
 \begin{array}{rcc}
R_{t}\, :\,\mathcal{C}\big([-T,T]; X^{-1}(\R^{2})\big)  &\longrightarrow&X^{-1}(\R^{2}) \\[3pt]
\dis  u&\longmapsto &\dis u(t,.),
 \end{array}
 \end{equation*}
is well defined and continuous.

Thus, for all $t\in [-T,T]$, $u_{N_{k}}(t)$ and $\wt{u}_{N_{k}}(t)$ have same distribution $(R_t)_\# \nu_{N_k}$. By Proposition \ref{4.4}, we obtain that this distribution is $\mu$.

Thus from \eqref{CV}  we deduce that 
\begin{equation}\label{LOI}
\mathscr{L}_{X^{-1}}( \wt{u}(t)) =\mu,\quad \forall\,t\in [-T,T].
\end{equation}
Let $k\geq 1$ and $t\in \R$ and consider the r.v. $X_{k}$ given by
\begin{equation*}
X_{k}=u_{N_{k}}(t) - R_0 (u_{N_{k}}(t)) +i \int_{0}^{t}\T_{N_{k}}( u_{N_{k}}   )ds. 
\end{equation*}
Define  $\widetilde{X}_{k}$ similarly to  $X_{k}$ with $u_{N_{k}}$ replaced with $\widetilde{u}_{N_{k}}$. Then by \eqref{loi}, 
$$\mathscr{L}_{\mathcal{C}_{T}X^{-1}}(\widetilde{X}_{N_{k}})=\mathscr{L}_{\mathcal{C}_{T}X^{-1}}(X_{N_{k}})=\delta_{0}.$$
 In other words, $\widetilde{X}_{k}=0$ ${\bf \widetilde{p}}$\,--\,a.s. and 
 $\widetilde{u}_{N_{k}}$ satisfies the following equation ${\bf  \widetilde{p}}$\,--\,a.s.
\begin{equation}\label{tilde}
\widetilde u_{N_{k}}(t) = R_0 (\widetilde u_{N_{k}}(t)) -i \int_{0}^{t}\T_{N_{k}}( \widetilde u_{N_{k}}   )ds .
 \end{equation}

We now show that we can pass to the limit $k\longrightarrow+\infty$ in \eqref{tilde} in order to show that $\widetilde{u}$ is ${\bf  \widetilde{p}}$\,--\,a.s. a   solution to~\eqref{CR} written in integral form as:
\begin{equation}\label{tilde2}
\widetilde u(t) = R_0 (\widetilde u(t)) -i \int_{0}^{t}\T( \widetilde u  )ds .
 \end{equation}

Firstly, from \eqref{CV} we deduce the convergence of the linear terms in equation \eqref{tilde} to those in \eqref{tilde2}. 
The following lemma gives the convergence of the nonlinear term.
 
\begin{lemm}\ph
Up to a sub-sequence, the following convergence holds true
\begin{equation*} 
\T_{N_{k}}(\wt{u}_{N_{k}}) \longrightarrow \T(\wt{u}),\quad \;\;\widetilde{\bf p}-\text{a.s. in}\;\; L^{2}\big([-T,T]; \H^{-\s}(\R^2)\big) . 
\end{equation*}
\end{lemm}

 \begin{proof}
 In order to simplify the notations, in this proof we drop the tildes and write $N_{k}=k$. Let $M\geq1$ and write 
  \begin{equation*}
 \T_k (u_{k}) - \T (u)= \big(   \T_k (u_{k}) -   \T (u_{k})\big) +\big(  \T (u_{k}) -  \T_M (u_{k}) \big)+\big(   \T_M (u_{k})-    \T_M (u)\big)+\big(     \T_M (u)- \T (u) \big).
 \end{equation*}
  To begin with, by continuity of the product in finite dimension, when $k\longrightarrow +\infty$
\begin{equation*}
  \T_M (  u_{k})  \longrightarrow     \T_M (  u) ,\quad \;\;\widetilde{\bf p}-\text{a.s. in}\;\; L^{2}\big([-T,T]; \H^{-\s}(\R^2)\big).
\end{equation*}
 We now deal with the other terms. It is sufficient to show  the convergence in the space $X:=L^{2}\big(\Omega \times[-T,T]; \H^{-\s}(\R^2)\big)$, since  the almost sure convergence  follows after exaction of a sub-sequence. \\
 By definition and the  invariance of $\mu$  we obtain
  \begin{eqnarray*}
\big\| \,    \T_M( u_k)-\T(u_k) \,\big\|^{2}_{X}&=&\int_{\mathcal{C}([-T,T];X^{-1})}\big\| \,  \T_{M}(v)-\T(v)\,\big\|^{2}_{L^{2}_{T}\H^{-\s}_{x}} \text{d}\nu_{k}(v)\\
&=&\int_{X^{-1}(\R^2)}\Big\| \,  \T_M\big(\Phi_{k}(t)(f)\big)-\T\big(\Phi_{k}(t)(f)\big)    \,\Big\|^{2}_{L^{2}_{T}\H^{-\s}_{x}} \text{d}\mu(f)\\
&=&\int_{X^{-1}(\R^2)}\big\| \,  \T_{M}(f)-\T(f)\big)\big\|^{2}_{L^{2}_{T}\H^{-\s}_{x}} \text{d}\mu(f)\\
&=&2T\int_{X^{-1}(\R^2)}\big\| \,  \T_{M}(f)-\T(f)\,\big\|^{2}_{\H^{-\s}_{x}} \text{d}\mu(f),
\end{eqnarray*}
which  tends to 0 uniformly in $k\geq 1$ when $M\longrightarrow +\infty$, according to Proposition \ref{Prop.cauchy}.

The term $\big\| \,   \T_M(u)-\T(u)\,\big\|_{X}$ is treated similarly. Finally, with the same argument we show
 \begin{equation*}
 \big\| \,  \T_k(u_k)-\T(u_k)\,\big\|_{X} \leq C  \big\| \,  \T_k(f)-\T(f)\,\big\|_{L^{2}_{\mu}\H^{-\s}_{x}},
 \end{equation*}
which tends to 0 when $k\longrightarrow +\infty$. This completes the proof.
  \end{proof}

\subsection{\bf Conclusion of the proof of Theorem \ref{thm2}} 
Define $\wt{f}=\wt{u}(0):=R_0(\tilde u)$. Then by \eqref{LOI}, ${\mathscr{L}_{X^{-1}}( \,\wt{f}\,) =\mu}$ and by the previous arguments, there exists $\wt{\Omega'}\subset \wt{\Omega}$ such that $\wt{\p}(\wt{\Omega'})=1$ and for each $\om'\in \wt{\Omega'}$, the random variable $\widetilde{u}$ satisfies  the equation
  \begin{equation}\label{dem} 
\begin{aligned}
&\widetilde{u}=  \wt{f} -i\int_0^t \T(  \widetilde{u}) dt, \quad   (t,x)\in \R\times \R^2.
\end{aligned}
\end{equation}

Set $\Sigma=\wt{f}(\Omega')$, then $\mu(\Sigma)=\wt{\p}(\wt{\Omega'})=1$. 
It remains to check that we can construct a global dynamics. Take a sequence $T_{N}\to +\infty$, and perform the previous argument for $T=T_{N}$. For all $N\geq 1$, let  $\Sigma_{N}$ be the corresponding set of initial conditions and  set $\Sigma=\cap_{N\in \N}\Sigma_{N}$. Then $\mu(\Sigma)=1$ and  for all  $\wt f\in \Sigma$, there exists 
 $$\wt u\in \mathcal{C}\big(\R\,; X^{-1}(\R^{2})\big),$$
which solves \eqref{dem}. This completes the proof of Theorem \ref{thm2}.

  \end{document}